\documentclass[a4paper,12pt]{article}
\usepackage{amssymb}
\usepackage{amsfonts}
\usepackage{amsthm}
\usepackage{amsmath}
\usepackage[all]{xy}
\usepackage[dvips]{graphicx}
\author{Nati Linial}
\title{A King in every two consecutive tournaments}
\author{{Yehuda Afek \thanks{The Blavatnik School of Computer Science, Tel-Aviv University,
                Israel 69978. afek@tau.ac.il}} \and{Eli Gafni \thanks{Computer Science Department, Univ. of California, LA, CA 95024, eli@cs.ucla.edu }}
                \and{Nati Linial \thanks{School of Computer Science and Engineering, Hebrew
                 University, Jerusalem 91904, Israel nati@cs.huji.ac.il}}}

\theoremstyle{plain}  \newtheorem{ther}{Theorem}
\theoremstyle{plain}  
\theoremstyle{remark} 
\theoremstyle{plain}  \newtheorem{cor}[ther]{Corollary}
\theoremstyle{remark} 
\theoremstyle{plain}  
\theoremstyle{plain}  

\begin{document} \maketitle
\date{}
\abstract{
We think of a tournament $T=([n], E)$ as a communication
network where in each round of communication processor $P_i$ sends its information to $P_j$, for
every directed edge $ij \in E(T)$.  By Landau's theorem (1953) there
is a King in $T$, i.e., a processor whose initial input reaches every other processor
in two rounds or less.  Namely, a processor $P_{\nu}$ such that
after two rounds of communication along $T$'s edges, the initial information of $P_{\nu}$
reaches all other processors.  Here we consider a more general
scenario where an adversary selects an arbitrary series of
tournaments $T_1, T_2,\ldots$, so that in each
round $s=1, 2, \ldots$, communication is governed by the corresponding tournament
$T_s$.  We prove that for every series of tournaments that
the adversary selects, it is still true that after two rounds of communication, the initial input of at least one
processor reaches everyone.

Concretely, we show that for every two tournaments $T_1, T_2$ there
is a vertex in $[n]$ that can reach all vertices via (i) A step in
$T_1$, or (ii) A step in $T_2$ or (iii) A step in $T_1$ followed by
a step in $T_2$. }

\section{Introduction}

In a study of computational models in distributed systems~\cite{ag12} the
following problem concerning $n$ communicating processors was posed.
Initially each processor $P_i$ has its own input data item, $m_i$.  Nature
selects a series of tournaments $T_1, T_2,\ldots$ on vertex set $[n]$,
and communication proceeds in rounds.  For every $s=1, 2, \ldots$
each processor $P_i$ communicates in round $s$ every data item that
has reached him so far to every processor $P_j$ with $ij \in E(T_s)$.
By an old theorem of Landau~\cite{lan53}, if $T_1 \equiv T_2$ then
after two rounds a {\em King} emerges. Namely, a processor $P_\nu$, such that $m_\nu$
has reached all processors.
Here we address the question how many rounds are required for a King to emerge
if in each round an arbitrary tournament is selected.
Surprisingly the answer is still $2$.

\section{Two rounds suffice}
Let $T_1=([n], E_1), T_2=([n], E_2)$ be two tournaments. The condition that data item $m_i$ reaches processor $P_j$ after two rounds of communication is denoted by $i \Rightarrow j$. Clearly, this is equivalent to
\[
i=j~~~ \vee~~~( i,j) \in E_1~~\vee~~(i, j) \in E_2~~~\vee~~~\exists k {\mbox~s.t.~} (i, k) \in E_1\mbox{~ and~} (k, j) \in E_2
\]

It is useful to note that the negation of this condition $i \not \Rightarrow j$ is equivalent to

\begin{equation}\label{condition}
i\neq j ~~\wedge ~~(j, i) \in E_1~~\wedge~~(j, i) \in E_2~~~\wedge~~~\Gamma_2(j) \supset \Gamma_1(i)
\end{equation}
where $\Gamma_{\delta}(x)$ is the set of out-neighbors of $x$ in tournament $T_{\delta}$.

\begin{ther}
\label{KingRB}
Let $T_1=([n], E_1), T_2=([n], E_2)$ be two tournaments. Then there is $\nu \in [n]$ such that $\nu \Rightarrow j$ for every $j$.
\end{ther}

\begin{proof}
By induction on $n$. The statement is easily verified for $n=3$.  Let $n$ be the smallest integer for which the theorem does not hold and let $T_1=([n], E_1), T_2=([n], E_2)$ be a counterexample.  By minimality of $n$, for every $n \ge j \ge 1$ there is some $n \ge i \ge 1$ such that $i \Rightarrow_j k$ for every $k \neq i, j$, where $\Rightarrow_j$ indicates that the relation is defined with respect to tournaments $T_1 \setminus \{j\}$ and $T_2 \setminus \{j\}$. When this happens we say that $i$ is {\em singled out} by $j$. Clearly $i \not \Rightarrow j$, or else the theorem holds with $\nu = i$, since $i \Rightarrow_j k$ clearly implies $i \Rightarrow k$. Consequently, no vertex is singled out more than once. We denote $\pi(j)=i$ and conclude that $\pi$ is a permutation on $[n]$, since $\pi(j)$ is defined for every $n \ge j \ge 1$ and $\pi$ is an injective mapping.

However, this is impossible. By Condition (\ref{condition}), every $j$ satisfies $|\Gamma_2(j)| > |\Gamma_1(\pi(j))|$. But $\sum_j |\Gamma_2(j)| = \sum_j |\Gamma_1(\pi(j))| = {n \choose 2}$, since $\pi$ is a permutation. This contradiction completes the proof.
\end{proof}

The same proof yields as well

\begin{cor}
Let $T_1=([n], E_1), T_2=([n], E_2)$ be two tournaments. Then there is $\mu \in [n]$, such that $i \Rightarrow \mu$ for every $i$.
\end{cor}
\begin{proof}

The same proof works. Just switch between $T_1, T_2$ and reverse all edges in the two tournaments.

\end{proof}

\section{A weaker corollary with a simpler proof}

A simpler proof for a weaker corollary of Theorem \ref{KingRB} was suggested by an anonymous referee.  This corollary is a generalization of \cite{lan53}, and is as follows:

\begin{cor}
\label{cor:king}
Let $T_R=([n], E_{red}), T_B=([n], E_{blue})$ be a red and a blue tournaments. Then there is $\nu \in [n]$, called king, such that for every $j$, $j$ is reachable from $\nu$ with a rainbow directed path of length at most $2$.
\end{cor} 

The critical difference is that in Theorem \ref{KingRB} $j$ must be reachable by either a blue edge, or by a red edge, or by a red edge followed by a blue edge (but a blue edge followed by a red edge does not count).  In Corollary \ref{cor:king} on the other hand, $j$ must be reachable by eiter of the above {\bf or} by a blue edge followed by a red edge. This weaker corollary is not strong enough for our application \cite{ag12}, where $T_R$ corresponds to a first round of message delivery and $T_B$ to a second round, where $T_R$ and $T_B$ are selected by an adversary. And by Theorem \ref{KingRB} regardless of the adversary choices there is a node whose information becomes known to everyone in two rounds. This does not follow from Corollary \ref{cor:king}.

\begin{proof} While it is a straightforward corollary of Theorem \ref{KingRB} here is a simpler proof of Corollary \ref{cor:king}.
Let $p$ be of the largest in-degree in either $T_R$ or $T_B$. Assume, $w.l.o.g$ that the largest in-degree is realized in $T_R$. Then delete $p$ and use induction. Let the resulting king be $k$. As the blue in-degree of $k$ is no larger than the red in-degree of $p$ (by the definition of $p$), there is a path of length $1$ from $k$ to $p$ or a blue-red path of length $2$ from $k$ to $p$.
\end{proof}

\end{document}